\documentclass[english,reqno]{amsart}
\usepackage{hyperref} 
\hypersetup{hidelinks,backref=true,pagebackref=true,hyperindex=true,colorlinks=true,citecolor=red,linkcolor=blue,breaklinks=true,urlcolor=ocre,bookmarks=true,bookmarksopen=false,pdftitle={Title},pdfauthor={Author}}

\usepackage{enumerate,bm}
\usepackage{amsopn,bm}
\usepackage{amsmath}
\usepackage{amssymb}
\usepackage{amsfonts}
\usepackage{amsbsy}
\usepackage{amscd,indentfirst,epsfig}
\usepackage{amsfonts,amsmath,latexsym,amssymb,verbatim,amsbsy}
\usepackage{amsthm} 
\numberwithin{equation}{section}

\AtBeginDocument{\let\oldtheo\theo
\renewcommand{\theo}{%
\oldtheo\hypertarget{\***@currentHref\endcsname}{}}}
 \def \leq {\leqslant}
\def \geq {\geqslant}
\def\ind#1{\lower5pt\hbox{$\scriptstyle #1$}}

\def \d {\mathrm{d}}

\def \R{\mathbb R}
\def\S{\mathbb S}

\def\N{\mathbb N}

\def\H{\mathcal H}

\def \vb {v_\star}

\def\d{\mathrm{d}}

\def\W {\mathbb{W}}

\def\Q{\mathcal{Q}}

\newtheorem{theo}{Theorem}[section]
\newtheorem{prop}{Proposition}[section]

\newtheorem{lem}{Lemma}[section]

\newtheorem{nb}{Remark}[section]
\def \leq {\leqslant}
\def \geq {\geqslant}

\newcommand{\verti}[1]{{\left\vert\kern-0.25ex\left\vert\kern-0.25ex\left\vert #1 
    \right\vert\kern-0.25ex\right\vert\kern-0.25ex\right\vert}}

\def \cR  {\mathcal{R}}

\numberwithin{equation}{section}

\begin{document}

\title{Uniform estimates on the Fisher information for solutions to Boltzmann and Landau equations}

\author{Ricardo J. {\sc Alonso}}

\address{Departamento de Matem\'{a}tica, PUC-Rio, Rua Marqu\^{e}s de S\~ao Vicente 225, Rio de Janeiro, CEP 22451-900, Brazil.} \email{ralonso@mat.puc-rio.br}

\author{V\'{e}ronique {\sc Bagland}}

\address{Universit\'{e} Clermont Auvergne, LMBP, UMR 6620 - CNRS,  Campus des C\'ezeaux, 3, place Vasarely, TSA 60026, CS 60026, F-63178 Aubi\`ere Cedex,
 France.}\email{Veronique.Bagland@math.univ-bpclermont.fr}

 \author{Bertrand {\sc Lods}}

 \address{Universit\`{a} degli
Studi di Torino \& Collegio Carlo Alberto, Department of Economics and Statistics, Corso Unione Sovietica, 218/bis, 10134 Torino, Italy.}\email{bertrand.lods@unito.it}

\maketitle

\begin{abstract} In this note we prove that, under some minimal regularity assumptions on the initial datum, solutions to the spatially homogenous Boltzmann and Landau equations for hard potentials uniformly propagate the Fisher information. The proof of such a result is based upon some explicit pointwise lower bound on solutions to Boltzmann equation and strong diffusion properties for the Landau equation. We include an application of this result related to emergence and propagation of exponential tails for the solution's gradient.  These results complement estimates provided in \cite{VLandau,villa, DVLandau, TVLandau}.\\



\end{abstract}


\section{Introduction}\noindent
The Fisher information functional was introduced in \cite{Fisher}
\begin{equation}\label{fisher}
\mathcal{I}(f):=4\int_{\R^{d}}\Big|\nabla \sqrt{f(v)}\Big|^{2}\d v
\end{equation}
as a tool in statistics and information theory. It revealed itself a very powerful tool to control regularity and rate of convergence for solutions to several partial differential equations. In particular, in the study of Fokker-Planck equation, the control of the Fisher information along the Orstein-Uhlenbeck semigroup is the key point for the exponential rate of convergence to equilibrium \cite{CT} in relative entropy terms. Variants of such an approach can be applied to deal with more general parabolic problems \cite{Mon}. For these kind of problems, the Fisher information turns out to play the role of a Lyapunov functional. 

Such techniques have also been applied in the context of general collisional kinetic equation. In particular, for the Boltzmann equation with Maxwell molecules, exploiting commutations between the Boltzmann collision operator and the Orstein-Uhlenbeck semigroup, the Fisher information serves as a Lyapunov functional for the study of the long time relaxation \cite{Toscani, CGT}.  In \cite{CC1, CC2, villa}, the Fisher information was applied for general collision kernels in relation to the entropy production bounds for the Boltzmann equation.  Later in \cite{vill}, ground breaking work related to the Cercignani's conjecture was made using the Fisher information and the ideas preceding such work. \\

The aim of the present contribution is to further investigate the properties of Fisher information along solutions to two important kinetic equations: the Boltzmann equation for hard potentials, under cut-off assumption, and the Landau equation for hard potentials. More specifically, we show here that, along solutions to Boltzmann or Landau equations for hard potentials, the Fisher information will remain uniformly bounded 
\begin{equation}\label{eq:Unif}
\sup_{t\geq0}\mathcal{I}(f(t))\leq C(f_0)<\infty\end{equation} under minimal assumption on the initial datum.  For the Boltzmann equation, this improves, under less restrictive conditions, the local in time estimate obtained in \cite{villa} which reads
\begin{equation*}
\mathcal{I}(f(t))\leq e^{c\,t}\big(2I(f_0) + c\,(1+t^{3})\big)\,,\quad \text{for some explicit }\, c>0\,.
\end{equation*}  
Notice that such a bound \eqref{eq:Unif} generalizes to hard potentials model the estimates given in \cite{CGT} relative to propagation of smoothness. For solutions to the Landau equation, it has been proved that, in the case of Maxwellian molecules, the Fisher information is nondecreasing \cite{VLandau1} as well.

As an application of the uniform propagation of the Fisher information, one can deduce that, for any $t_{0} >0$, 
\begin{equation*}
\sup_{t\geq t_0>0}\int_{ \R^{d} }\big| \nabla f(t,v) \big| e^{c\,|v|^{\gamma}} \d v \leq C(f_0,t_{0})<\infty\,,\quad \text{for some explicit }\, c>0\,,
\end{equation*}  
in a relatively simple manner (relatively to \cite{alogam} for example). The techniques to prove the bound \eqref{eq:Unif} differ completely for the study of Boltzmann and Landau equations. For the Boltzmann equation, we exploit the \textit{appearance of pointwise exponential lower bounds} for solutions obtained in \cite{pulwen} whereas, for the Landau equation, we use the instantaneous regularizing effect to control, for time $t \geq t_{0} >0$ the Fisher information by Sobolev regularity bounds while, for small time $0 < t < t_{0}$, the Fisher information is controlled thanks to \emph{new energy estimates} for solutions to the Landau equation. 

\subsection{Notations}

Let us introduce some useful notations for function spaces. For any  $p \geq1$ and $q \geq0$, we define the space $L^{p}_{q}(\R^d)$ through the norm
$$\|f\|_{L^{p}_{q}}:=\left(\int_{\R^{d}}|f(v)|^{p}\langle v\rangle^{pq}\d v\right)^{1/p},$$
i.e. $L^{p}_{q}(\R^d)=\{f\::\R^{d} \to \R\;;\,\|f\|_{L^{p}_{q}} < \infty\}$ where, for $v \in \R^{d}$, $\langle v\rangle=\sqrt{1+|v|^{2}}$. We also define,  for $k \in \N$,
$$\W^{k,p}_{q}(\R^d)=\left\{f \in L^{p}_{q}(\R^d)\;;\;\partial_{v}^{\beta}f \in L^{p}_{q}(\R^d) \:\forall |\beta| \leq k\right\}$$
with the usual norm,
$$\|f\|_{\W^{k,p}_{q}}^{p}=\sum_{|\beta| \leq k}\|\partial_{v}^{\beta}f\|_{L^{p}_{q}}^{p},$$
where, for any multi-index $\beta=(\beta_{1},\ldots,\beta_{d}) \in \N^{d}$, $|\beta|=\sum_{i=1}^{d}\beta_{i}$ and $\partial^{\beta}_{v}=\partial_{v_{1}}^{\beta_{1}}\ldots\partial_{v_{d}}^{\beta_{d}}.$
We set $H^k_q(\R^d)=\W^{k,2}_q(\R^d)$ and also define $L^1_{\log}(\R^d)$ as 
$$L^1_{\log}(\R^d)=\left\{ f\in L^1(\R^d)\;; \;\int_{\R^d}|f(v)| \, |\log(|f(v)|)|\,\d v <\infty\;\right\}.$$

\subsection{The Boltzmann equation} 

  Let us now enter into the details by considering the solution $f(t,v)$ to the Boltzmann equation
\begin{equation}\label{eq:BE} 
\partial_{t}f(t,v)=\Q(f,f)(t,v)\,,\qquad v\in\mathbb{R}^{d}\,.
\end{equation}
We consider kernels satisfying $\|b\|_{L^{1}(\mathbb{S}^{d-1})}<\infty$, thus, it is possible to write the collision operator in gain and loss operators
\begin{equation*}
\Q(f,g)=\Q^{+}(f,g)-g\,\cR(f)\,,
\end{equation*}
where the collision operator is given by
\begin{align*}
\Q^{+}(f,g)(v)&:=\int_{\S^{d-1}\times \R^{d}}b(\cos\theta)|v-v_{\star}|^{\gamma}\,f(v'_{\star})\,g(v')\d v_{\star}\d \sigma\,,\\
\cR(f)(v)&:=\int_{\S^{d-1}\times \R^{d}}b(\cos\theta)|v-v_{\star}|^{\gamma}\,f(v_{\star})\d v_{\star}\d \sigma = \|b\|_{L^{1}(\mathbb{S}^{d-1})}\big(f\ast |u|^{\gamma}\big)(v)\,.
\end{align*}
We will consider hard potentials $\gamma\in(0,1]$.  Also, for technical simplicity, we restrict ourself to $d\geq3$. 
 \begin{theo}{\textit{\textbf{(Uniform propagation of the Fisher information)}}}\phantomsection\label{main} Let $b\in L^{2}(\mathbb{S}^{d-1})$ be the angular scattering kernel, $d\geq3$ and $\gamma\in(0,1]$.  Assume also that the initial datum $f_{0}\geq0$ satisfies 
\begin{equation*}
f_{0} \in L^{1}_{\eta}(\R^{d}) \cap L^{2}_{\mu}(\R^{d}) \cap H^{\frac{(5-d)^{+}}{2}}_{\nu}(\R^{d}), 
\end{equation*}
for some ${ \nu > 3 + \gamma + \frac{d}{2} }$, ${ \mu \geq \nu + 1 + \frac{\gamma}{2}}$, ${ \eta \geq \mu + d }$ and 
$$\int_{\R^d} f_0(v)\, v\,\d v=0,\qquad\qquad \mathcal{I}(f_{0}) < \infty.$$
Then, the unique solution $f(t)\geq0$ of \eqref{eq:BE} satisfies
\begin{equation*}
\sup_{t\geq0}\,\mathcal{I}(f(t)) \leq C\,,
\end{equation*}
for some positive constant $C$ depending on $\mathcal{I}(f_{0})$ and the $L^{1}_{\eta}\cap L^{2}_{\mu} \cap H^{\frac{(5-d)^{+}}{2}}_{\nu}$-norm of $f_{0}.$
\end{theo}
\begin{nb}
For $d \geq 5,$ the result holds for $f_{0} \in L^{1}_{\eta}(\R^{d}) \cap L^{2}_{\mu}(\R^{d})$ for any $\mu\geq \nu + 1 + \gamma$ and ${ \eta \geq \mu + d }$.  Of course, $\mathcal{I}(f_0)$ must be finite and we must have $\int_{\R^d} f_0(v)\, v\,\d v=0$.
\end{nb}
\begin{nb}
If the reader is willing to accept more regularity in the initial data, say $f_0\in H^{2}_{\nu}(\mathbb{R}^{d})$ for some $\nu > \frac{d}{2}$, then Theorem \ref{main} remains valid for $b\in L^{1}(\mathbb{S}^{d-1})$ using the propagation of regularity given in \cite{alogam} and the control of the Fisher information using the $H^{2}_{\nu}(\mathbb{R}^{d})$ norm, see \cite[Lemma 1]{TVLandau}.
\end{nb}
\subsection{The Landau equation} As mentioned earlier,  we also investigate the case of solutions to the homogeneous Landau equation. Recall that such an equation reads
\begin{equation}\label{Le0}
\partial_{t}f = \Q_{L}(f,f)\,, \qquad v\in\mathbb{R}^{d}\,.
\end{equation}
The collision operator is defined as
\begin{equation}\label{Lce-1}
\Q_{L}(f,f)(v)=\nabla\cdot\int_{\R^{d}} A(v-v_\star)\Big(f(v_{\star})\nabla f(v) - f(v)\nabla f(v_{\star})\Big) \d v_{\star}\,
\end{equation}
where the matrix $A(z)=(A_{ij}(z))_{i,j=1,\ldots,d}$ is given by
\begin{equation*}
A_{ij}(z)=\left(\delta_{ij}-\frac{z_{i}z_{j}}{|z|^{2}}\right)\Phi(z), \qquad \Phi(z):=|z|^{2+\gamma}\,.
\end{equation*}
We concentrate the study in the hard potential case $\gamma\in(0,1]$.  We refer to \cite{DVLandau} for a methodical study of the Landau equation in this setting.  The Landau equation can be written in the form of a nonlinear parabolic equation:
\begin{equation}\label{Le-2}
\partial_{t}f(t,v) - \nabla\cdot\big(a(v)\nabla f(t,v)  - b(v)\,f(t,v)\big) = 0\,,
\end{equation}
where the matrix $a(v)$ and the vector $b(v)$ are given by
\begin{equation*}
a:=A * f\,,\qquad b:=\nabla\cdot A* f\,.
\end{equation*}
The minimal conditions that will be required on the initial datum $f_{0}$ are finite mass, energy and entropy
\begin{equation*}
m_0: = \int_{\R^{d}} f_0(v)\d v < +\infty \,,\quad E_0: = \int_{\R^{d}} |v|^{2}f_{0}(v)\d v < +\infty \,,\quad H_0:=\int_{\R^{d}} f_0(v)\log f_0(v)\d v < +\infty\,.  
\end{equation*}
For technical reasons, to assure conservation of energy, a moment higher than 2 is assumed as well.  In this situation, \cite[Proposition 4]{DVLandau} asserts that the equation is uniformly elliptic, that is,
\begin{equation*}
a(v) \xi \cdot \xi \geq a_0\,\langle v \rangle^{\gamma}\, |\xi|^{2}\,,\qquad \forall\; v \in \R^{d}, \; \xi \in \R^{d}
\end{equation*}
for some positive constant $a_0:=a_0(m_0,E_0,H_0).$ Under these assumptions, the Cauchy theory, including infinite regularization and moment propagation, has been developed in \cite{DVLandau,DVLandau2}.   As in the Boltzmann case, the Fisher information have been used for the analysis of convergence towards equilibrium, see for instance \cite{DVLandau2, TVLandau, VLandau}, and also for analysis of regularity, see \cite{DLandau}.  The idea is to establish an inequality of the form
\begin{equation*}
\int_{\R^{d}} |\nabla \sqrt{f}|^{2} \d v \leq C\big(\mathcal{D}(f)+1\big),
\end{equation*} 
with constant $C$ depending only on $m_0,E_0,H_0$, which are the physical conserved quantities, and where $\mathcal{D}(f)$ denotes the entropy production associated to $\Q_{L}$, i.e.
$$\mathcal{D}(f)=-\int_{\R^{d}}\Q_{L}(f,f)\log f\d v.$$ 
Since, along solutions to the Landau equation $f(t)=f(t,v)$ it holds that
\begin{equation*}
0\leq \int^{t}_0 \mathcal{D}(f(s))\d s\leq C_{\mathcal{D}}(m_0,E_0,H_0,t)\,,
\end{equation*}
such inequality leads to estimate on the \textit{time integrated} Fisher information.  Then, one uses Sobolev inequality to obtain control on the entropy or a higher norm.\\

\noindent For the Fisher information itself, at least for the hard potential case, the following result follows.
\begin{theo}\label{main2} 
Assume that the initial datum $f_0\geq0$ has finite mass $m_{0}$, energy $E_{0}$ and entropy $H_{0}$ and satisfies in addition
\begin{equation}\label{hme2}
\int_{\R^d}\langle v \rangle^{2} f_0(v)\log f_0(v) \d v <+\infty\,,\qquad \int_{\R^d}\langle v \rangle^{2 + \gamma +\epsilon} f_0(v)\d v < +\infty\,,
\end{equation}
for some $\epsilon>0$.  Assume moreover that $\mathcal{I}(f_{0}) < \infty.$ Then, there exists a weak solution $f(t)=f(t,v)$ to \eqref{Le0} with initial datum $f_0$ satisfying
\begin{equation*}
\sup_{t\geq0}\mathcal{I}(f(t)) \leq C^{0}_{F}\,,
\end{equation*}
where the constant $C^{0}_{F}$ depends on $m_0, E_0, H_{0}$, the quantities in \eqref{hme2}, and the initial Fisher information.
\end{theo}
\begin{nb}
If we also assume that $f_0\in L^2_s(\R^3)$ with $s> (5\gamma+15)/2$ then there exists a unique weak solution to \eqref{Le0} with initial datum $f_0$ (see \cite[Theorem7]{DVLandau}. Consequently, Theorem \ref{main2} is valid for any weak solution to \eqref{Le0} with initial datum $f_0$.
\end{nb}
The rest of the document is divided in three sections, Section 2 is devoted to the proof of Theorem \ref{main} and Section 3 is concerned with the proof of Theorem \ref{main2}.  The final section is an Appendix where the reader will find helpful facts about Boltzman (Appendix A.) and Landau (Appendix B.) equations that will be needed along the arguments. 
\section{Proof of Theorem \ref{main}}
In order to prove Theorem \ref{main}, we consider in all this section a solution $f(t)=f(t,v)$ to the Boltzmann equation \eqref{eq:BE} that conserves mass, momentum, and energy. One has first the following lemma.
\begin{lem}\label{lemma1}
The Fisher information of $f(t,\cdot)$ satisfies
\begin{align}\label{fish}
\begin{split}
\dfrac{\d }{\d t}\mathcal{I}(f(t)) & = -2\int_{\R^{d}}\log f(t,v)\,\,\Delta_{v}\Q^{+}(f,f)(t,v)\d v - 4\int_{\R^{d}}\left|\nabla \sqrt{f(t,v)}\right|^{2}\cR(f)(t,v)\d v\\
& - 2\int_{\R^{d}} \nabla  f(t,v)  \cdot \nabla \cR(f)(t,v)\, \d v-\int_{\R^{d}}\left|\nabla \log{f(t,v)}\right|^{2}\Q^{+}(f,f)(t,v)\d v\,.
\end{split}
\end{align}
\end{lem}
\begin{proof} One first notices that $g_{i}(t,v):=\partial_{v_i}\sqrt{f(t,v)}$ satisfies
\begin{multline*}
\partial_{t} g_{i}(t,v)=\partial_{v_i}\left(\frac{1}{2\sqrt{f(t,v)}}\Q(f,f)(t,v)\right)=-\frac{1}{2f(t,v)}g_{i}(t,v)\,\Q(f,f)(t,v)+\frac{1}{2\sqrt{f(t,v)}}\partial_{v_i}\Q(f,f)(t,v)\\
=-\frac{1}{2f(t,v)}g_{i}(t,v)\,\Q^{+}(f,f)(t,v)+\frac{1}{2}g_{i}(t,v)\cR(f)(t,v)\\
+\frac{1}{2\sqrt{f(t,v)}}\partial_{v_i}\Q^{+}(f,f)(t,v)-\frac{1}{2\sqrt{f(t,v)}}\partial_{v_i}\left(f(t,v)\cR(f)(t,v)\right).
\end{multline*}
Multiplying by $g_{i}(t,v)$ and integrating over $\R^{d}$ we get
\begin{multline*}
\dfrac{1}{2}\dfrac{\d}{\d t}\left\|g_{i}(t)\right\|_{L^2}^{2}=-\frac{1}{2}\int_{\R^{d}}\frac{g_{i}^{2}(t,v)}{f(t,v)}\,\Q^{+}(f,f)(t,v)\d v+\frac{1}{2}\int_{\R^{d}}g_{i}^{2}(t,v)\cR(f)(t,v)\d v\\
+\frac{1}{2}\int_{\R^{d}}\frac{g_{i}(t,v)}{\sqrt{f(t,v)}}\partial_{v_i}\Q^{+}(f,f)(t,v)\d v-\frac{1}{2}\int_{\R^{d}}\frac{g_{i}(t,v)}{\sqrt{f(t,v)}}\partial_{v_i}\left(f(t,v)\cR(f)(t,v)\right)\d v.
\end{multline*}
Noticing that
\begin{equation*}
\frac{g_{i}^{2}(t,v)}{f(t,v)}=\left(\frac{\partial_{v_i}f(t,v)}{2f(t,v)}\right)^{2}=\frac{1}{4}\left(\partial_{v_i}\log f(t,v)\right)^{2}\,,
\end{equation*}
and
\begin{align*}
\frac{g_{i}(t,v)}{\sqrt{f(t,v)}}\partial_{v_i}\left(f(t,v)\cR(f)(t,v)\right)&=\frac{g_{i}(t,v)\partial_{v_i}f(t,v)}{\sqrt{f(t,v)}}\cR(f)(t,v)+g_{i}(t,v)\sqrt{f(t,v)}\partial_{v_i}\cR(f)(t,v)\\
&=\frac{(\partial_{v_i}f(t,v))^{2}}{2f(t,v)}\cR(f)(t,v) +\frac{1}{2}\partial_{v_i}f(t,v)\partial_{v_i}\cR(f)(t,v)\,,
\end{align*}
we get that
\begin{align*}
\dfrac{1}{2}\dfrac{\d}{\d t}\left\|g_{i}(t)\right\|_{L^2}^{2}&= - \frac{1}{8}\int_{\R^{d}}\left(\partial_{v_i}\log f(t,v)\right)^{2}\,\Q^{+}(f,f)(t,v)\d v+\frac{1}{2}\int_{\R^{d}}g_{i}^{2}(t,v)\cR(f)(t,v)\d v\\
&\hspace{-0.3cm}+\frac{1}{4}\int_{\R^{d}} \partial_{v_i}\log f(t,v)\,\partial_{v_i}\Q^{+}(f,f)(t,v)\d v-\frac{1}{4}\int_{\R^{d}}\frac{(\partial_{v_i}f(t,v))^{2}}{f(t,v)}\cR(f)(t,v) \d v\\
&\hspace{0.3cm} - \frac{1}{4}\int_{\R^{d}}\partial_{v_i}f(t,v)\partial_{v_i}\cR(f)(t,v)\d v\,.
\end{align*}
Using an integration by part in the third integral, and since $\dfrac{(\partial_{v_i}f(t,v))^{2}}{4f(t,v)}=g_{i}^{2}(t,v)$, this results easily in
\begin{multline*}
\dfrac{\d}{\d t}\left\|g_{i}(t)\right\|_{L^2}^{2}=-\frac{1}{4}\int_{\R^{d}}\left(\partial_{v_i}\log f(t,v)\right)^{2}\,\Q^{+}(f,f)(t,v)\d v-\int_{\R^{d}}g_{i}^{2}(t,v)\cR(f)(t,v)\d v\\
-\frac{1}{2}\int_{\R^{d}} \log f(t,v)\,\partial^{2}_{ii}\Q^{+}(f,f)(t,v)\d v
-\frac{1}{2}\int_{\R^{d}}\partial_{v_i}f(t,v)\partial_{v_i}\cR(f)(t,v)\d v
\end{multline*}
which yields the desired result after adding in $i=1,2,\cdots d$.
\end{proof}
All terms in \eqref{fish} are relatively easy to estimate with exception, perhaps, of the term involving $\Delta Q^{+}(f,f)$.
\begin{lem}\label{lemma2} 
Let $f(t)\geq0$ be a sufficiently smooth solution of the Boltzmann equation.  Then, for any $\varepsilon>0$ 
\begin{equation*}
\int_{\R^{d}}\big|\log f(t,v)\big|\,\big|\Delta_{v}\Q^{+}(f,f)(t,v)\big|\d v \leq C(\varepsilon,d)\Big(c_{\varepsilon}(t) + \|f(t)\|_{L^{2}}\Big)\Big(\|f(t)\|^{2}_{H^{s}_{\eta_{1}}} + \|f(t)\|^{2}_{L^{1}_{\eta_{2}}}\Big)\,,
\end{equation*}
where $c_{\varepsilon}(t):=C_{\varepsilon}\big(1+\log^{+}(1/t)\big)$ for some universal constant $C_{\varepsilon} >0$, and
\begin{equation*}
{\eta_{1}:=\frac{6+2\gamma+d+3\,\varepsilon}{2}, \qquad \eta_{2}:=\frac{4+2\gamma+ d + 3\,\varepsilon}{2}}, \qquad s=\frac{(5-d)^{+}}{2}\leq 1\,.
\end{equation*}
\end{lem} 
\begin{proof}
Using Theorem \ref{propo:lower}, we get that
\begin{equation*}
\int_{\R^{d}}\big|\log f(t,v)\big|\,\big|\Delta_{v}\Q^{+}(f,f)(t,v)\big|\d v \leq \int_{\R^{d}}\Big(c_{\varepsilon}(t)\langle v \rangle^{2+\varepsilon} + f(t,v)\Big)\,\big|\Delta_{v}\Q^{+}(f,f)(t,v) \big|\d v\,.
\end{equation*}
Thus, 
\begin{equation*}
\int_{\R^{d}}\big|\log f(t,v)\big|\,\big|\Delta_{v}\Q^{+}(f,f)(t,v)\big|\d v \leq c_{\varepsilon}(t)\,\|\Q^{+}(f(t),f(t))\|_{\W^{2,1}_{2+\varepsilon}} + \|f(t)\|_{L^{2}}\|\Q^{+}(f(t),f(t))\|_{H^{2}}.\end{equation*}
Using the interpolation
\begin{equation*}
\|h\|_{L^{1}_{s}} \leq C_{\tau}(d)\|h\|_{L^{2}_{s+\tau}} \qquad \forall\, \tau > d/2, \quad s \in \R
\end{equation*}
for constant $C_{\tau}(d)=\|\langle \cdot\rangle^{-\tau}\|_{L^{2}}$, we get that for $\tau=\frac{d+\varepsilon}{2}$,
\begin{equation*}
\|\Q^{+}(f(t),f(t))\|_{\W^{2,1}_{2+\varepsilon}} \leq C_{\frac{d+\varepsilon}{2}}(d)\|\Q^{+}(f(t),f(t))\|_{H^{2}_{2+\frac{3\varepsilon+d}{2}}}\,.
\end{equation*}
This results in 
\begin{equation*}
\int_{\R^{d}}\big|\log f(t,v)\big|\,\big|\Delta_{v}\Q^{+}(f,f)(t,v)\big|\d v \leq C_{\frac{d+\varepsilon}{2}}(d)\Big(c_{\varepsilon}(t) + \|f(t)\|_{L^{2}}\Big)\|\Q^{+}(f(t),f(t))\|_{H^{2}_{2+\frac{3\varepsilon+d}{2}}}\,.
\end{equation*}
 Now, using Theorem \ref{regularite} we can estimate the last term and get
\begin{equation}\label{log} 
\int_{\R^{d}}\big|\log f(t,v)\big|\,\big|\Delta_{v}\Q^{+}(f,f)(t,v)\big|\d v \leq C(\varepsilon,d)\Big(c_{\varepsilon}(t) + \|f(t)\|_{L^{2}}\Big)\Big(\|f(t)\|^{2}_{H^{s}_{\eta_{1}}} + \|f(t)\|^{2}_{L^{1}_{\eta_{2}}}\Big)\,
\end{equation}
with $\eta_{1},\, \eta_{2}$ and $s$ as defined in the statement of the lemma.
\end{proof}
\begin{proof}[Proof of Theorem \ref{main}] We start with \eqref{fish} and neglect the nonpositive last term in the right side.  It follows that
\begin{align*}
\dfrac{\d }{\d t}&\mathcal{I}(f(t)) \leq -2\int_{\R^{d}}\log f(t,v)\,\,\Delta_{v}\Q^{+}(f,f)(t,v)\d v \\
& - 4\int_{\R^{d}}\left|\nabla \sqrt{f(t,v)}\right|^{2}\cR(f)(t,v)\d v
- 2\int_{\R^{d}} \nabla  f(t,v)  \cdot \nabla \cR(f)(t,v)\, \d v\,.
\end{align*}
Additionally, thanks to \eqref{eq:low}, one has $\mathcal{R}(f)(v)\geq \kappa_0 \langle v \rangle^{\gamma}$.  And due to integration by parts and \eqref{eq:laplace}
\begin{align*}
-2\int_{\R^{d}} \nabla  f(t,v)  \cdot \nabla & \cR(f)(t,v)\, \d v = 2\int_{\R^{d}}f(t,v)\,\Delta_{v} \cR(f)(t,v)\d v \\
&\leq C_{d,\gamma}\|b\|_{L^1(\S^{d-1})}\|f\|_{L^{1}}\Big(\|f\|_{L^{1}} + \|f\|_{H^{\frac{(4-d)^{+}}{2}}}\Big)\,.
\end{align*}
Therefore,
\begin{align*}
\dfrac{\d }{\d t}\mathcal{I}(f(t)) + \kappa_{0}\,\mathcal{I}(f(t)) &\leq 2\int_{\R^{d}}\big|\log f(t,v)\big|\,\big|\Delta_{v}\Q^{+}(f,f)(t,v)\big|\d v + 2\int_{\R^{d}}f(t,v)\,\Delta_{v} \cR(f)(t,v)\d v\\
&\leq C(\varepsilon,d,b)\Big(c_{\varepsilon}(t) + \|f(t)\|_{L^{2}}+ \|f(t)\|_{L^{1}} \Big)\Big(\|f(t)\|^{2}_{H^{s}_{\eta_{1}}} + \|f(t)\|^{2}_{L^{1}_{\eta_{2}}} + 1 \Big)\,,
\end{align*}
where we used, in addition to previous estimates, Lemma \ref{lemma2} for the second inequality.  Here $\eta_{1},\, \eta_{2}$, and $s$ are those defined in such lemma.\\

Under our assumptions on $f_{0}$ and for a suitable choice of $\varepsilon >0$ small enough, the $L^{1}_{\eta_{2}}$ and $H^{1}_{\eta_{1}}$ norms of $f(t)$ are uniformly bounded, see Theorems \ref{propo:mom} and \ref{propo:h1}.  Thus, we obtain that, for such choice of $\varepsilon >0$, it holds
\begin{equation*}
\dfrac{\d }{\d t}\mathcal{I}(f(t)) +\kappa_{0}\,\mathcal{I}(f(t))  \leq C(f_0)(1+\log^{+}(1/t)), \qquad t >0\,.
\end{equation*}
Using that the mapping $t\mapsto 1+\log^{+}(1/t)$ is integrable at $t=0$, a direct integration of this differential inequality implies that $\sup_{t\geq0}\mathcal{I}(f(t)) \leq \mathcal{I}(f_0) +C(f_0)<\infty$. This proves the result.
\end{proof}
A consequence of this result is the exponentially weighted generation/propagation of the solution's gradient.  Indeed, one knows thanks to \cite{cmpde} that $\|f(t)e^{c\min\{1,t\}|v|^{\gamma}}\|_{L^{1}}\leq C(f_0)$ for some sufficiently small $c>0$ and constant $C(f_0)$ depending only on mass and energy.  Then,
\begin{align*}
\int_{\mathbb{R}^{d}}\big| \nabla f(t,v) \big|  e^{\frac{c}{2}\min\{1,t\}|v|^{\gamma}} \text{d}v &= 2\int_{\mathbb{R}^{d}}\big| \nabla \sqrt{f} \big|  \sqrt{f}\,e^{\frac{c}{2}\min\{1,t\}|v|^{\gamma}} \text{d}v \\
&\leq \mathcal{I}(f(t))^{\frac{1}{2}} \big\|f(t)e^{c\min\{1,t\}|v|^{\gamma}}\big\|^{\frac{1}{2}}_{L^{1}}\leq C(f_0)\,.
\end{align*}

\section{Proof of Theorem \ref{main2}}

In this section, we prove the uniform in time estimate on the Fisher information for solution to the Landau equation.  The strong diffusion properties of Landau make the Fisher information more suited to this equation than to Boltzmann.

\medskip

\noindent We assume in all this section that $f(t)=f(t,v)$ is a solution to \eqref{Lce-1} with initial datum $f_{0}(v)$ with mass $m_0$, energy $E_0$.  We also assume that $f_0$ has finite entropy $H_0$. We shall exploit the parabolic form of the Landau equation that we recall here again 
\begin{equation}\label{e1}
\partial_{t}f - \nabla\cdot(a\nabla f) + \nabla\cdot(b f)=0\,,
\end{equation}
for $a:=a(v)$ symmetric positive definite matrix and $b:=b(v)$ vector. Recall that, according to \eqref{eq:elliptic}, the matrix $a=a(t,v)$ is uniformly elliptic, i.e.
\begin{equation*}
a(t,v)\xi \cdot \xi \geq a_{0}\langle v\rangle^{\gamma}|\xi|^{2}, \qquad \forall\; v \in \R^{3},\:\xi \in \R^{3},\;t \geq 0\,.
\end{equation*} 
Multiplying the equation by $\log f$ and integrating
\begin{equation*}
\frac{\text{d}}{\text{d}t}\int_{\R^3}f\log f \d v+ \int_{\R^3}a\nabla f \cdot \frac{\nabla f}{f} \d v+ \int_{\R^3}(\nabla\cdot b)\,f\d v=0 \,.
\end{equation*}
We recall, see \eqref{eq:bc}, that
\begin{equation*}
\big|(\nabla\cdot b)(v)\big|\leq B(m_0,E_0)\langle v \rangle^{\gamma}\,, 
\end{equation*}
and, using \eqref{eq:elliptic}
\begin{equation*}
\int_{\R^3}a\nabla f \cdot \frac{\nabla f}{f} \d v \geq a_{0}\int_{\R^3}\langle v \rangle^{\gamma}\,\nabla f \cdot \frac{\nabla f}{f} \d v = 4 a_0\int_{\R^3}\langle v \rangle^{\gamma}\, \big|\nabla\sqrt{f}\big|^{2}\,\d v. 
\end{equation*}
As a consequence,
\begin{equation*}
\frac{\text{d}}{\text{d}t}\int_{\R^3}f\log f \d v +  4 a_0\int_{\R^3}\langle v \rangle^{\gamma}\,\big|\nabla\sqrt{f}\big|^{2}  \d v\leq \tilde{B}(m_0,E_0)\,.
\end{equation*}
Integrating in time
\begin{equation*}
4 a_0\int^{t}_{0}\d s\int_{\R^3}\langle v \rangle^{\gamma}\,\big|\nabla\sqrt{f(s,v)}\big|^{2}\d v \leq  \int_{\R^3}f_0\log f_0 \d v- \int_{\R^3}f(t,v)\log f(t,v)\d v + t\,\tilde{B}(m_0,E_0)\,,\quad t>0\,.
\end{equation*}
Since
\begin{equation*}
\sup_{t\geq0}\Big|\int_{\R^3}f(t,v)\log f(t,v)\d v\Big|\leq H(m_0,E_0,H_0)\,,
\end{equation*}
we just proved the first part of the following proposition.
\begin{prop}\label{p1}
For a solution $f(t)=f(t,v)$ to the Landau equation one has
\begin{equation}\label{eq:fishnoweight}
4\int^{t}_{0}\d s\int_{\R^3}\langle v \rangle^{\gamma}\,\big|\nabla\sqrt{f(s,v)}\big|^{2}\d v \leq C(m_0,E_0,H_{0})\big(1+t\big),\qquad t>0\,.
\end{equation} 
Moreover, given $k>0$ and $\epsilon>0$, if we assume the initial datum $f_0$ to be such that
\begin{equation}\label{hme}
\int_{\R^3}\langle v \rangle^{k} f_0(v)\log f_0(v) \d v <+\infty\,,\qquad \int_{\R^3}\langle v \rangle^{k + \gamma+\epsilon} f_0(v)\d v < +\infty\,,
\end{equation}
then
\begin{equation*}
4\int^{t}_{0}\d s\int_{\R^3}\langle v \rangle^{k+\gamma}\,\big|\nabla\sqrt{f(s,v)}\big|^{2}\d v \leq C_{k}(m_0,E_0,H_{0})\big(1+t\big),\qquad t>0\,,
\end{equation*} 
for some positive constant $C_{k}$ depending on the mass $m_{0}$, the energy $E_{0}$, the entropy $H_{0}$ and the quantities \eqref{hme}.
\end{prop}
\begin{proof} We already proved \eqref{eq:fishnoweight}, it remains to prove the weighted Fisher information statement.  For this, we multiply \eqref{e1} by $\langle \cdot \rangle^{\gamma}\log f(t,\cdot)$ and, integrating over $\R^{3}$ we obtain
\begin{multline*}
\dfrac{\d}{\d t}\int_{\R^{3}}\langle v\rangle^{k}f(t,v)\log f(t,v)\d v=\dfrac{\d}{\d t}\int_{\R^{3}}f(t,v)\langle v\rangle^{k}\d v - \int_{\R^{3}}\langle v\rangle^{k}\nabla \cdot (b(v)f(t,v)) \log f(t,v)\d v\\
+\int_{\R^{3}} \langle v \rangle^{k}\nabla \cdot \left(a(v)\nabla f(t,v)\right)\log f(t,v)\d v.\end{multline*} Note that integrations by parts lead to  
\begin{align*}
\int_{\R^3}- \nabla\cdot (a \nabla f)\,\langle v \rangle^{k}\log f \d v &= \int_{\R^3}\langle v \rangle^{k}a\nabla f \cdot \frac{\nabla f}{f}\d v - k\int_{\R^3}(f\log f - f) \nabla\cdot \big(a \langle v \rangle^{k-2}v\big)\d v\\
&\geq 4a_0\int_{\R^3}\langle v \rangle^{k+\gamma}\big| \nabla\sqrt{f} \big|^{2}\d v - A_0\, k\, \int_{\R^3}\langle v \rangle^{k+\gamma}\big(f |\log f| + f \big)\d v\,.
\end{align*}
The latter inequality follows by using \eqref{eq:elliptic} and the fact that
\begin{equation*}
\big| \nabla\cdot \big(a \langle v \rangle^{k-2}v\big) \big| \leq A_0\langle v \rangle^{k+\gamma}\,.
\end{equation*}
Similarly,
\begin{align*}
\bigg|\int_{\R^3}\nabla\cdot(b f)\,\langle v \rangle^{k}\log f \, \d v \,\bigg| &= \bigg| \int_{\R^3}\nabla\cdot\big(\langle v \rangle^{k}b\big)\,f \,\d v  - k\int_{\R^3}(f\log f) \, b\cdot \langle v \rangle^{k-2}v \, \d v \,\bigg|\\
&\leq B_0\int_{\R^3}\langle v \rangle^{k+\gamma}f\,  \d v + B_0\,k\,\int_{\R^3}\langle v \rangle^{k+\gamma} f |\log f| \,\d v .
\end{align*}
We control the integral with $f | \log f |$ using Lemma \ref{lem:LlogL} with $\delta >0$ small enough.  It follows that
\begin{equation}\label{eq:LlogL}
\frac{\d}{\d t}\int_{\R^3}\langle v \rangle^{k}f(t,v)\,\log f(t,v)\d v + 2a_0\|\langle v \rangle^{\frac{k+\gamma}{2}}\nabla \sqrt{f(t,v)}\|^{2}_{2} \leq\dfrac{\d}{\d t}\int_{\R^3}\langle v \rangle^{k} f(t,v)\d v + \tilde{C}_{k}\,.
\end{equation}
for some positive constant $\tilde{C}_{k}$ depending only on $\sup_{t\geq 0}\|f(t)\|_{L^{1}_{k+\gamma+\varepsilon}}$ for some arbitrary $\varepsilon >0$. Integrating between $0$ and $t$ the previous equation, we get 
\begin{multline*}
\int_{\R^3}\langle v \rangle^{k}f(t,v)\,\log f(t,v)\d v + 2a_0 \int_0^t \|\langle v \rangle^{\frac{k+\gamma}{2}}\nabla \sqrt{f(s,v)}\|^{2}_{2}\,\d s\\ \leq \int_{\R^3}\langle v \rangle^{k}f_0(v)\,\log f_0(v)\d v + \int_{\R^3}\langle v \rangle^{k} f(t,v)\d v + \tilde{C}_{k} t \,.
\end{multline*}
The first integral in the left-hand side has no sign but it can be handled thanks to \eqref{eq:Hp}. The result follows from here using propagation of the moment $k+\gamma+\varepsilon$.
\end{proof}
One notices that, for solutions of the Landau equation for hard potentials, the Fisher information emerges as soon as $t>0$. This result immediately follows from the following lemma.
\begin{lem}\label{lem:large} Let $f(t)$ be the weak solution to \eqref{Le0} with initial datum $f_{0}$ given by \cite[Theorem 5]{DVLandau}. For any $t_{0} >0$, there is $C_{t_{0}} >0$ depending only on $m_{0},E_{0}$ and $H_{0}$ such that
$$\sup_{t\geq t_{0}}\mathcal{I}(f(t)) \leq C_{t_{0}}.$$
\end{lem}
\begin{proof} The result is a direct consequence of the following link between the Fisher entropy and weighted Sobolev norm, see \cite[Lemma 1]{TVLandau} and \cite[Theorem 5]{DVLandau}: there is $C >0$ such that
$$\mathcal{I}(f) \leq C\,\|f\|_{H^{2}_{\frac{d+1}{2}}} \qquad \forall\, f \in H^{2}_{\frac{d+1}{2}}.$$
We conclude then with 
Lemma \ref{lem:soboL}.\end{proof}
With this result at hand, it remains to study the question about the behaviour of the Fisher information at $t=0$.  To this end, we prove the following lemma.
\begin{lem} Let $f=f(t,v)$ be a solution to \eqref{e1} with initial datum $f_{0}$ with mass $m_{0},$ energy $E_{0}$ and entropy $H_{0}$ satisfying \eqref{hme2}.
Introduce for $i=1,\ldots,d$
\begin{equation*}
g_{i}=\partial_{v_{i}}\sqrt{f}\,,\quad a^{i}=\partial_{v_{i}}a\,,\quad b^{i}=\partial_{v_{i}}b\,,\quad g:=\nabla\sqrt{f}\,.
\end{equation*}
Then, there exist $A_{0}$ and $C_{1}$ depending only on $m_{0},E_{0},H_{0}$ and the quantities \eqref{hme2} such that
\begin{multline}\label{e2}
\frac{\d}{\d t}\int_{\R^3}|g_{i}(t,v)|^{2} \d v+ a_0\int_{\R^3}\langle v \rangle^{\gamma}\Big| \nabla g_{i}(t,v) - \frac{g_{i}(t,v)}{\sqrt{f(t,v)}}\,g(t,v) \Big|^{2}\d v\\
\leq A_0\int_{\R^3}\langle v \rangle^{\gamma+2}\big|\nabla \sqrt{f(t,v)}\big |^{2}\d v + C_{1}\,.
\end{multline}
\end{lem}
\begin{proof}
With the notations of the lemma and recalling that $a=a(t,v)$ is symmetric, one can compute
\begin{align*}
-\partial_{v_{i}}\Big(\frac{1}{\sqrt{f}}\nabla\cdot(a\nabla f)\Big) 
& = - 2 \,\partial_{v_{i}}\Big(\frac{1}{\sqrt{f}}\nabla\cdot(a g  \sqrt{f})\Big) \\
&= - 2\,\nabla\cdot(a\nabla g_i) + 2\,\frac{g_{i}}{f}\, g \cdot ag -\frac{4}{\sqrt{f}}\,\nabla g_{i} \cdot ag - 2\nabla\cdot(a^{i}g) - \frac{2}{\sqrt{f}}\,g\cdot a^{i}g\,. 
\end{align*}
We also have
\begin{equation*}
\partial_{v_{i}}\Big(\frac{1}{\sqrt{f}} \nabla\cdot (b f) \Big) = \nabla\cdot (b\, g_{i}) + b\cdot\nabla g_{i} + \nabla\cdot(b^{i} \sqrt{f}) + b^{i}\cdot g\,.
\end{equation*}
As a consequence, after some integration by parts, the Dirichlet terms are computed as
\begin{align*}
\int_{\R^3}-\partial_{v_{i}}\Big(\frac{1}{\sqrt{f}}\nabla\cdot(a\nabla f)\Big) g_{i}\d v &= 2\int_{\R^3}\Big(a\nabla g_{i}\cdot \nabla g_{i} + \frac{g^{2}_{i}}{f}\,g\cdot ag - \frac{2g_{i}}{\sqrt{f}}\nabla g_{i}\cdot ag\Big)\d v\\
&\hspace{2cm}+2\int_{\R^3}\Big(a^{i}g\cdot\nabla g_{i} - \frac{g_{i}}{\sqrt{f}}\,g\cdot a^{i}g\Big)\d v\\
&= 2\int_{\R^3}\Big| \sqrt{a}\Big( \nabla g_{i} - \frac{g_{i}}{\sqrt{f}}\,g \Big) \Big|^{2}\d v + 2\int_{\R^3}a^{i}g \cdot \Big( \nabla g_{i} - \frac{g_{i}}{\sqrt{f}}\,g \Big)\d v\,.
\end{align*} 
Here $\sqrt{a}=\sqrt{a(t,v)}$ is the unique positive definite symmetric square root of $a(t,v)$.  In addition,
\begin{align*}
\int_{\R^3}\partial_{v_{i}}\Big(\frac{1}{\sqrt{f}} \nabla\cdot (b f) \Big)\,g_{i} \d v=  - \int_{\R^3} b^{i} \sqrt{f} \cdot \Big( \nabla g_{i} - \frac{g_{i}}{\sqrt{f}}\, g \Big)\d v\,.
\end{align*}
Consequently, we can find an energy estimate for $g_{i}$.  Indeed, multiplying the Landau equation \eqref{e1} by $1/\sqrt{f}$, differentiating in $v_{i}$, multiplying by $g_i$ and integrating in velocity, it follows that
\begin{multline*}
\frac{\text{d}}{\text{d}t}\int_{\R^3}|g_{i}(t,v)|^{2}\d v + 2\int_{\R^3}\Big| \sqrt{a(t,v)}\Big( \nabla g_{i}(t,v) - \frac{g_{i}(t,v)}{\sqrt{f(t,v)}}\,g(t,v) \Big) \Big|^{2}\d v \\
+ 2\int_{\R^3}a^{i}(t,v)g(t,v) \cdot \Big( \nabla g_{i}(t,v) - \frac{g_{i}(t,v)}{\sqrt{f(t,v)}}\,g(t,v) \Big)\d v\\
 - \int_{\R^3} \sqrt{f(t,v)}\, b^{i}(t,v)   \cdot \Big( \nabla g_{i}(t,v) - \frac{g_{i}(t,v)}{\sqrt{f(t,v)}}\, g(t,v) \Big)\d v =0 \,.
\end{multline*}
We proceed estimating each term, starting for the absorption term
\begin{align*}
\int_{\R^3}\Big| \sqrt{a}\Big( \nabla g_{i} - \frac{g_{i}}{\sqrt{f}}\,g \Big) \Big|^{2}\d v &= \int_{\R^3}a \Big( \nabla g_{i} - \frac{g_{i}}{\sqrt{f}}\,g \Big)\cdot \Big( \nabla g_{i} - \frac{g_{i}}{\sqrt{f}}\,g \Big)\d v\\
&\geq a_0\int_{\R^3}\langle v \rangle^{\gamma}\Big| \nabla g_{i} - \frac{g_{i}}{\sqrt{f}}\,g \Big|^{2}\d v\,.
\end{align*}
For the latter two terms we use Young's inequality $2|ab|\leq \epsilon a^{2} + \epsilon^{-1}b^{2}$ with $\epsilon = 2\,a_0/3$ to obtain
\begin{align*}
\frac{\d}{\d t}\int_{\R^3}|g_{i}|^{2}\d v + a_0\int_{\R^3}\langle v \rangle^{\gamma}\Big| \nabla g_{i} - \frac{g_{i}}{\sqrt{f}}\,g \Big|^{2}\d v\leq \tfrac{3}{2a_0}\int_{\R^3}\langle v \rangle^{-\gamma}\big|a^{i}\, g\big |^{2}\d v + \tfrac{3}{4a_0}\int_{\R^3}\langle v \rangle^{-\gamma}\big| b^{i} \sqrt{f} \big|^{2}\d v\,.
\end{align*}
We recall that $|b^{i}|\leq B(m_0,E_0)\langle v \rangle^{\gamma}$, therefore,
\begin{equation*}
\int_{\R^3}\langle v \rangle^{-\gamma}\big| b^{i} \sqrt{f} \big|^{2}\d v \leq B(m_0,E_0)^{2}\int_{\R^3}\langle v \rangle^{\gamma}\,f \d v\leq C_{1}(m_0,E_0)\,.
\end{equation*}
Also, $|a^{i}|\leq A(m_0,E_0)\langle v \rangle^{\gamma+1}$.  As a consequence,
\begin{equation*}
\int_{\R^3}\langle v \rangle^{-\gamma}\big|a^{i}\, g\big |^{2}\d v \leq A(m_0,E_0)\int_{\R^3}\langle v \rangle^{\gamma+2}|g|^{2}\d v = A\int_{\R^3}\langle v \rangle^{\gamma+2}\big| \nabla \sqrt{f} \big|^{2}\d v\,.
\end{equation*}
This gives the result.\end{proof}

\begin{proof}[Proof of Theorem \ref{main2}]
For short time, say $t\in [0,1]$, integrate \eqref{e2} in time and use Proposition \ref{p1} with $k=2$.  Then, we can invoke Lemma \ref{lem:large} with $t_{0}=1$ to estimate $\mathcal{I}(f(t))$ for $t \geq 1.$\end{proof}
\subsection{Exponential moments for the Landau equation}
In \cite[Section 3]{DVLandau} emergence and propagation of polynomial moments have been obtained  for the Landau equation and, more recently \cite[Section 3.2]{CLandau} develops the propagation of exponential moments for soft potentials.  The starting point is the weak formulation for the equation 
\begin{equation}\label{eme-1}
\frac{\d}{\d t}\int_{\mathbb{R}^{3}} f(t,v)\varphi(v)\,\d v = 2\sum_{j}\int_{\mathbb{R}^{3}}f(t,v)\,b_{j}\,\partial_{v_j}\varphi(v) \,\d  v + \sum_{i,j}\int_{\mathbb{R}^{3}}f(t,v)\,a_{ij}\partial^2_{v_iv_j}\varphi(v)\,\d v.
\end{equation}
Exponential moments can be easily studied in a similar fashion by choosing $\varphi(v) = e^{\lambda\langle v \rangle^{s}}$ with positive parameters $\lambda,\,s$ to be determined.  We note that, for such a choice,
\begin{equation*}
\partial_{v_j}\varphi(v) = \lambda s\,e^{\lambda\langle v \rangle^{s}}\langle v \rangle^{s-2} v_{j}\,,\qquad \partial^2_{v_iv_j}\varphi(v) = \lambda s\,e^{\lambda\langle v \rangle^{s}}\Big((s-2)\langle v \rangle^{s-4}v_{i}v_{j} + \langle v \rangle^{s-2}\delta_{ij} + \lambda s \langle v \rangle^{2(s-2)}v_{i}v_{j}\Big)\,.
\end{equation*}
Thus, resuming the computations given in \cite[pg. 201]{DVLandau} one gets
\begin{align*}
\frac{\d}{\d t}\int_{\mathbb{R}^{3}}f(t,v)e^{\lambda \langle v \rangle^{s}}\d v &= \lambda s\int_{\mathbb{R}^{3}}\int_{\mathbb{R}^{3}}f(t,v)\,f(t,v_{\star})|v-v_{\star}|^{\gamma}e^{\lambda \langle v \rangle^{s}}\langle v \rangle^{s-2}\\
& \times \Big( - 2|v|^{2} + 2|v_{\star}|^{2} + \big(|v|^{2}|v_{\star}|^{2} - (v\cdot v_{\star})^{2}\big)\big( (s-2)\langle v \rangle^{-2} + \lambda s\langle v \rangle^{s-2}\Big)\d v\,\d v_{\star}\,.
\end{align*}
At this point, we choose $0<s<2$ and thanks to the Young inequality 
$\lambda s\langle v \rangle^{s}\langle v_{\star}\rangle^{2} \leq \frac{s}{2}\langle v \rangle +\frac{2}{2-s} (\lambda s)^{\frac{2}{2-s}} \langle v_{\star}\rangle^{\frac{4}{2-s}}$,
we have
\begin{align*}
-2|v|^{2} + &2|v_{\star}|^{2} +  \big(|v|^{2}|v_{\star}|^{2} - (v\cdot v_{\star})^{2}\big)\big( (s-2)\langle v \rangle^{-2} + \lambda s\langle v \rangle^{s-2}\big) \leq   -2|v|^{2} + 2|v_{\star}|^{2} + \lambda s\langle v \rangle^{s}|v_{\star}|^{2}\\
&\leq -\frac{4-s}{2}\langle v \rangle^{2} + 2\langle v_{\star} \rangle^{2} + \frac{(2-s)}{2}(\lambda s)^{\frac{2}{2-s}}\langle v_{\star} \rangle^{\frac{4}{2-s}}\leq  - \langle v \rangle^{2} + 2\langle v_{\star} \rangle^{2} + C_{s}\lambda ^{\frac{2}{2-s}}\langle v_{\star} \rangle^{\frac{4}{2-s}}\,.
\end{align*}
Thus, using Lemma \ref{lem:momentL}, we get
\begin{multline}\label{peme-2}
\frac{\d}{\d t}\int_{\mathbb{R}^{3}}f(t,v)e^{\lambda\langle v \rangle^{s}}\d v \leq \lambda s\int_{\mathbb{R}^{3}}f(t,v)\,e^{\lambda \langle v \rangle^{s}}\langle v \rangle^{s+\gamma}\Big( - c + C\langle v \rangle^{-2}\Big)\d v\\
\leq \lambda s\int_{\mathbb{R}^{3}}f(t,v)\,e^{\lambda \langle v \rangle^{s}}\langle v \rangle^{s+\gamma}\Big( - \frac{c}{2} + \,C\,\text{1}_{\{|v|\leq r\}}\Big)\d v
\end{multline}
where $c>0$ depends on $m_0, E_0$.  Meanwhile,
\begin{equation*}\label{pemC}
C=2\,\sup_{t\geq 0}\|f(t)\|_{L^{1}_{2+\gamma}} + C_{s}\lambda ^{\frac{2}{2-s}}\sup_{t\geq 0}\| f(t) \|_{ L^{1}_{ \frac{4}{2-s} + \gamma} }\,,\quad\text{and}\quad r:=r(C,c,\gamma)\,.
\end{equation*}
This proves a propagation result for exponential moments.
\begin{prop}\label{plandau} Fix $s\in(0,\gamma]$ and assume that $f_0$ belongs to $L^{1}_{2+\gamma}(\R^{3})$.  Then, for the solution $f(t,v)$ of the Landau equation with initial datum $f_{0}$ given by \cite[Theorem 5]{DVLandau} there exists some $\beta:=\beta_{s,\gamma}\geq1$ such that
\begin{equation*}
\sup_{t\geq0}\int_{\mathbb{R}^{3}} f(t,v)e^{\min\{1,t^{\beta}\}\langle v \rangle^{s}}  \d v \leq C(f_0)\, \qquad \hfill \text{{(Emergence of tails)}}.
\end{equation*}
Fix $s\in(0,2)\,$, $\lambda >0\,$, and assume that $\int_{\mathbb{R}^{3}} f_0\, e^{\lambda \langle v \rangle^{s}}\d v<\infty$.  Then, for the solution $f(t,v)$ of the Landau equation with initial datum $f_{0}$ given by \cite[Theorem 5]{DVLandau} it follows that
\begin{equation*}
\sup_{t\geq 0}\int_{\mathbb{R}^{3}}f(t,v)e^{\lambda\langle v \rangle^{s}} \d v \leq C_{\lambda ,s}(f_0)\,\: \qquad \hfill \text{(Propagation of tails)}.
\end{equation*}
\end{prop}
\begin{proof}
For the emergence of the exponential tail we assume $t\in(0,1)$ and take $\varphi(t,v)=e^{t^{\beta}\langle v \rangle^{s}}$ with $s\in(0,2)$ and $\beta>0$ to be chosen.  We repeat the steps leading to estimate \eqref{peme-2} to obtain
\begin{equation}\label{peme-3}
\frac{\d}{\d t}\int_{\mathbb{R}^{3}}f(t,v)e^{t^{\beta}\langle v \rangle^{s}}\d v \leq t^{\beta}s\int_{\mathbb{R}^{3}}f(t,v)\,e^{t^{\beta}\langle v \rangle^{s}}\langle v \rangle^{s+\gamma}\Big( - c + C(t) \,\langle v \rangle^{-2} + \frac{\beta}{st}\langle v \rangle^{-\gamma}\Big)\d v\,.
\end{equation}
The constant $c>0$ depends on $m_0, E_0$ whereas $C(t)$ is given by 
$$C(t)=2\, \|f(t)\|_{L^{1}_{2+\gamma}} + C_{s}t^{\frac{2\beta}{2-s}}\| f(t) \|_{ L^{1}_{ \frac{4}{2-s} + \gamma} }\,.$$
Similarly to the Boltzmann equation, one can prove with the techniques given in \cite[Section 3]{DVLandau} that $\|f\|_{L^1_k}\lesssim t^{-k/\gamma}$.  Therefore, choosing 
\begin{equation*}
\beta = \frac{4 + (2-s)\gamma}{(4-s)\gamma} > 1\,,
\end{equation*}
we guarantee that $C(t)\lesssim t^{-\beta}$.  Thus,
\begin{equation*}
- c + C\langle v \rangle^{-2} + \frac{\beta}{st}\langle v \rangle^{-\gamma} \leq -c + \frac{C_{1}}{t^{\beta}}\langle v \rangle^{-\gamma} \leq -\frac{c}{2} + \frac{C_{1}}{t^{\beta}}\,\text{1}_{\{ |v|\leq t^{-\beta/\gamma}\,r \}}\,,
\end{equation*}
where the radius $r:=r(C_{1},c)$ is independent of time.  Therefore,
\begin{equation*}
\frac{\d}{\d t}\int_{\mathbb{R}^{3}}f(t,v)e^{t^{\beta}\langle v \rangle^{s}}\d v \leq s\,C_{1}\,e^{t^{\beta(1-s/\gamma)}\langle r\rangle^s}\int_{\mathbb{R}^{3}}f(t,v)\,\langle v \rangle^{s+\gamma}\d v\leq \tilde{C}(f_0)\,,\quad 0< s \leq \gamma\,.
\end{equation*}
This proves the generation of the exponential tail.
\end{proof}
As previously expressed for the Boltzmann equation, the propagation/generation of the Fisher information and the exponential moments imply the propagation/generation of the exponential moments for the gradient of solutions.  For any $s\in(0,\gamma]$
\begin{align*}
\int_{\mathbb{R}^{3}}\big| \nabla f(t,v) \big|  e^{\frac{\min\{1,t^{\beta}\}}{2}\langle v \rangle^{s}} \d v &= 2\int_{\mathbb{R}^{3}}\big| \nabla \sqrt{f} \big|  \sqrt{f}\,e^{\frac{\min\{1,t^{\beta}\}}{2}\langle v\rangle^{s}} \d v \\
&\leq \mathcal{I}(f(t))^{\frac{1}{2}}\big\|f(t)e^{\min\{1,t^{\beta}\}\langle v\rangle^{s}}\big\|^{\frac{1}{2}}_{L^{1}}\leq C(f_0)\,.
\end{align*}

\appendix

\section{Regularity estimates for the Boltzmann equation}
We include here some classical results in the theory of the homogeneous Boltzmann equation.  We use them in the core of this note.
\begin{theo}\label{propo:lower}
Let $b\in L^{1}(\mathbb{S}^{d-1})$ be the scattering kernel and $\gamma\in(0,1]$.  Let $0\leq f_{0}\in L^{1}_{2}(\mathbb{R}^{d})\cap L^{1}_{log}(\mathbb{R}^{d})$ be the initial data.  Then, the unique solution to \eqref{eq:BE} satisfies: for any $\varepsilon >0$ there exists $C_{\varepsilon} >0$ such that
\begin{equation*}
\left|\log f(t,v)\right|  \leq C_{\varepsilon}\left(1+\log^{+}(1/t)\right)\langle v\rangle^{2+\varepsilon} + f(t,v), \qquad v \in \R^{d}, \qquad t >0.
\end{equation*}
\end{theo}
\begin{proof} The proof relies on \cite[Theorem 1.1 \& Lemma 3.1]{pulwen} and follows after keeping track of the time dependence of the constants involved.  A similar argument was made to prove \cite[Theorem 3.5]{ABL}.
\end{proof}

\begin{theo}(See \cite[Theorem 4.2]{Wenn} and \cite[Lemma 8]{cmpde})\label{propo:mom}
Let $b\in L^{1}(\mathbb{S}^{d-1})$ be the scattering kernel, $\gamma\in(0,1]$, and assume $0\leq f_{0} \in L^{1}_{2}(\R^{d})$. Then, for every $k > 0$ there exists a constant $C_{k} \geq 0$ depending only on $k,\,b$, and the initial mass and energy of $f_0$, such that
\begin{equation*}
m_{k}(t):=\int_{\R^{d}}f(t,v)|v|^{k}\d v \leq C_{k}\max(1, t^{-k/\gamma}) \qquad \text{ for } t > 0.
\end{equation*}
If, in addition, $m_{k}(0)<\infty$ then
\begin{equation*}
\sup_{t\geq 0}m_{k}(t) \leq C_{k}\,,
\end{equation*}
for some constant $C_{k}$ depending only on $k,\, b$, the mass and energy of $f_0$, and $m_{k}(0)$.
\end{theo}
\begin{lem}\label{LappA1} Let $b\in L^{1}(\mathbb{S}^{d-1})$ be the scattering kernel and $\gamma\in(0,1]$.  Let $0\leq f(t) \in L^{1}_{2+\varepsilon}(\R^{d})$, with $\varepsilon >0$, be such that for some $C \geq c >0$
\begin{equation*}
C \geq \int_{\R^{d}}f(t,v)\langle v\rangle^{2}\d v \geq  c, \qquad \int_{\R^{d}}f(t,v)\,v\, \d v = 0\,.
\end{equation*}
Then, there exists $\kappa_{0}$ depending on $C,\,c,$ $b$ and $\sup_{t\geq0}\|f(t)\|_{L^{1}_{2+\varepsilon}}$ such that
\begin{equation}\label{eq:low}
\cR(f)(v) \geq \kappa_{0}\langle v\rangle^{\gamma}.
\end{equation}
Moreover,
\begin{equation}\label{eq:laplace}
0\leq \Delta_{v} \cR(f)(v) \leq C_{d,\gamma}\|b\|_{L^1(\S^{d-1})}\Big(\|f\|_{L^{1}} + \|f\|_{H^{\frac{(4-d)^{+}}{2}}}\Big).
\end{equation}
\end{lem}
\begin{proof} The lower bound \eqref{eq:low} has been established in \cite[Lemma 2.1]{alogam}. Let us focus on the second point by directly computing
\begin{equation*}
\Delta_{v}\cR(f)(v)=\mathrm{div}_{v}\left(\nabla \cR(f)(v)\right)=\gamma\,\|b\|_{L^1(\S^{d-1})}\int_{\R^{d}}\mathrm{div}_{v}\left((v-\vb)|v-\vb|^{\gamma-2}\right)f(\vb)\d \vb.
\end{equation*}
Since ${\mathrm{div}_{v}\left((v-\vb)|v-\vb|^{\gamma-2}\right)=(d+\gamma-2)|v-\vb|^{\gamma-2}}$, we get
\begin{align*}
0&\leq\Delta_{v}\cR(f)(v)=\gamma\,(d + \gamma - 2)\|b\|_{L^1(\S^{d-1})}\int_{\R^{d}}|v-\vb|^{\gamma-2}f(\vb)\d \vb\\
&\leq \gamma\,(d + \gamma - 2)\|b\|_{L^1(\S^{d-1})}\bigg(\Big(\int_{\R^{d}}\big| f(\vb) \big|^{\frac{d}{d-2}}\d \vb\Big)^{\frac{d-2}{d}} \Big(\int_{\{|\vb|\leq1\}}\big| \vb \big|^{\frac{d(\gamma-2)}{2}}\d \vb \Big)^{\frac{2}{d}}+ \int_{\R^{d}}f(\vb)\d \vb\bigg)\\
&\leq C_{d,\gamma}\Big(\|f\|_{L^{1}} + \|f\|_{H^{\frac{(4-d)^{+}}{2}}}\Big)\,.
\end{align*}
For the last inequality we used the Sobolev embedding valid for $d\geq3$.
\end{proof}
\begin{theo}(See \cite[Corollary 1.1]{AlCaGa} and \cite[Theorem 4.1]{MoVi})\label{propo:L2}
Let $b\in L^{1}(\mathbb{S}^{d-1})$ be the scattering kernel and $\gamma\in(0,1]$.  For a fixed $\eta \geq 0$ assume that
\begin{equation*}
0\leq f_{0} \in L^{1}_{{\eta + d }}(\R^{d}) \cap L^{2}_{\eta}(\R^{d})\,.
\end{equation*}
Then, 
\begin{equation*}
\sup_{t\geq 0}\|f(t)\|_{L^{2}_{\eta}}<\infty.
\end{equation*}
\end{theo} 
\begin{theo}(See \cite[Theorem 2.1]{bouchut} and \cite[Theorem 3.5]{MoVi})\phantomsection\label{regularite} Let $b\in L^{2}(\mathbb{S}^{d-1})$ be the scattering kernel and $\gamma\in(0,1]$. Then, for all $s \geq 0$ and all $\eta \geq 0$, it holds
\begin{equation*}
\|\Q^{+}(g,f)\|_{H^{s+\frac{d-1}{2}}_{\eta}} \leq C_{d}\left(\|g\|_{H^{s}_{\eta+1+\gamma}}\,\|f\|_{H^{s}_{\eta+1+\gamma}}+\|g\|_{L^{1}_{\eta+\gamma}}\|f\|_{L^{1}_{\eta+\gamma}}\right).
\end{equation*}
for some positive constant $C_{d}$ depending only on the dimension $d$.
\end{theo}
\begin{theo}{(See \cite[Theorem 4.2]{MoVi})}\phantomsection\label{propo:h1}
Let $b\in L^{2}(\mathbb{S}^{d-1})$ be the scattering kernel and $\gamma\in(0,1]$.  Let $\eta\geq 0$ and assume that the initial datum $f_{0}$ satisfies
\begin{equation*}
f_{0} \in  L^{1}_{{\eta+1+\gamma/2+d}}(\R^{d})\cap L^{2}_{{\eta+1+\gamma/2}}(\R^{d})\cap H^{1}_{\eta}(\R^{d})\,.
\end{equation*}
Then, the unique solution $f(t,v)$  to \eqref{eq:BE} with initial condition $f_0$ satisfies
\begin{equation*}
\sup_{t \geq 0}\|f(t)\|_{H^{1}_{\eta}}:=C_{\eta} <\infty\,.
\end{equation*}
\end{theo}
\begin{proof} Set $g(t,v)=\nabla f(t,v)$ so that $\partial_{t}g(t,v)=\nabla \Q(f,f)(t,v)$. 
Applying the inner product of such equation with $\langle v \rangle^{2\eta}g(t,v)$ and integrating over $\R^{d}$ we get that  
\begin{equation*}
\frac{1}{2}\dfrac{\d}{\d t}\|g(t)\|_{L^{2}_{\eta}}^{2} =\int_{\R^{d}}\langle v\rangle^{2\eta}g(t,v)\cdot\nabla\Q^{+}(f,f)(t,v)\d v- \int_{\R^{d}}\langle v\rangle^{2\eta}g(t,v)\cdot\nabla\Q^{-}(f,f)(t,v)\d v.
\end{equation*}
Notice that 
\begin{equation*}
\nabla\Q^{-}(f,f)(t,v)=g(t,v)\cR(f(t,\cdot))(v) + f(t,v)\nabla\cR(f(t,\cdot))(v)
\end{equation*}
so that, after using \eqref{eq:low},  
\begin{equation*}\int_{\R^{d}}\langle v\rangle^{2\eta}g(t,v)\cdot\nabla\Q^{-}(f,f)(t,v)\d v
\geq \kappa_{0}\|g(t)\|_{L^{2}_{\eta+\gamma/2}}^{2} +\int_{\R^{d}}\langle v \rangle^{2\eta}f(t,v)g(t,v)\cdot\nabla \cR(f)(t,v)\d v.
\end{equation*}
Thus, 
\begin{align}\label{eq:phi}
\begin{split}
\frac{1}{2}\dfrac{\d}{\d t}\|g(t)\|_{L^{2}_{\eta}}^{2} + \kappa_{0}\|g(t)\|_{L^{2}_{\eta+\gamma/2}}^{2} \leq \|g(t)\|_{L^{2}_{\eta+\gamma/2}}&\|\nabla\Q^{+}(f(t),f(t))\|_{L^{2}_{\eta-\gamma/2}}\\
& - \int_{\R^{d}}\langle v \rangle^{2\eta}f(t,v)g(t,v)\cdot\nabla\cR(f)(t,v)\d v\,.
\end{split}
\end{align}
Since
\begin{equation*}
\big| \nabla\mathcal{R}(f) \big| \leq \gamma\|b\|_{L^{1}(\mathbb{S}^{d-1})}C_{d}\Big(\|f\|_{L^{1}} + \|f\|_{L^{2}} \Big)\,,
\end{equation*}
we estimate this last integral as 
\begin{align}\label{eq:ga1}
\begin{split}
\bigg|\int_{\R^{d}}\langle v \rangle^{2\eta}f(t,v) g(t,v)\cdot\nabla\cR(f)(t,v)\d v\bigg|\leq C(f_0)\|b\|_{L^{1}(\mathbb{S}^{d-1})}\int_{\R^{d}}\langle v \rangle^{2\eta}|g(t,v)|\,f(t,v)\d v\\
\leq C(f_0)\|b\|_{L^{1}(\mathbb{S}^{d-1})}\|f(t)\|_{L^{2}_{\eta}}\|g(t)\|_{L^{2}_{\eta}}\leq C(f_{0},b)\|g(t)\|_{L^{2}_{\eta}}\,.
\end{split}
\end{align}
Using \eqref{eq:ga1} and Theorem \ref{regularite} in \eqref{eq:phi}, we obtain that
\begin{equation*}
\frac{1}{2}\dfrac{\d}{\d t}\|g(t)\|_{L^{2}_{\eta}}^{2} + \kappa_{0}\|g(t)\|_{L^{2}_{\eta+\gamma/2}}^{2}\leq C_{3}\|g(t)\|_{L^{2}_{\eta+\gamma/2}}\Big(\|f(t)\|_{L^{1}_{\eta+\gamma/2}}^{2} + \|f(t)\|_{L^{2}_{\eta+1+\gamma/2}}^{2}\Big) + C(f_{0},b)\|g(t)\|_{L^{2}_{\eta}}\,.
\end{equation*}
Thus, since
\begin{equation*}
\sup_{t \geq 0}\left(\|f(t)\|_{L^{2}_{\eta+1+\gamma/2}}+\|f(t)\|_{L^{1}_{\eta+\gamma/2}}\right)\leq C(f_0)
\end{equation*}
according to Theorems \ref{propo:mom} and \ref{propo:L2} and our hypothesis on $f_0$, it follows that
\begin{equation*}
\frac{1}{2}\dfrac{\d}{\d t}\|g(t)\|_{L^{2}_{\eta}}^{2} + \kappa_{0}\|g(t)\|_{L^{2}_{\eta+\gamma/2}}^{2}\leq C(f_0,b)\,\|g(t)\|_{L^{2}_{\eta+\gamma/2}}, \qquad \forall t >0\,,
\end{equation*}
which readily gives that
\begin{equation*}
\sup_{t\geq0}\|g(t)\|_{L^{2}_{\eta}}\leq \max\Big\{\|g_0\|_{L^{2}_{\eta}},\tfrac{C(f_0,b)}{\kappa_0} \Big\}\,.
\end{equation*}
This together with the propagation of $\|f\|_{L^{2}_{\eta}}$ proves the result.
\end{proof}

\section{Regularity estimates for the Landau equation}\label{Landau}

We collect here known results, extracted from \cite{DVLandau} about the regularity of solutions to the Landau equation \eqref{Lce-1}. We begin with classical estimate related to the matrix $A(z)$. For $ (i,j) \in  [\hspace{-0.7mm} [  1,3 ]\hspace{-0.7mm}] ^2 $, 
we recall that
$$A(z)= \left(A_{i,j}(z)\right)_{i,j} \quad \mbox{ with }
\quad A_{i,j}(z) 
= |z|^{\gamma+2} \,\left( \delta_{i,j} -\frac{z_i  z_j}{|z|^2} \right),$$
and introduce
$$B_i(z)= \sum_k \partial_k A_{i,k}(z) = -2 \,z_i \, |z|^{\gamma}.$$
 For any $f \in L^{1}_{2+\gamma}(\R^{3})$, we define then the matrix-valued mapping $a(v)=A\ast f(v)$ and the vector-valued mapping $b(v)=(b_{i}(v))_{i}$ with
$$b_{i}(v)=B_{i} \ast f, \qquad \forall v \in \R^{3},\qquad i=1,\ldots,3.$$
One has the following \cite[Proposition 4]{DVLandau}:
\begin{lem}\label{lem:abA}  There is a positive constant $a_{0}$ depending only on $m_{0},E_{0},H_{0}$ such that
\begin{equation}\label{eq:elliptic}
a(v)\xi \cdot \xi=\sum_{i,j=1}^{3}a_{ij}(v)\xi_{i}\xi_{j} \geq a_0\,\langle v \rangle^{\gamma}\, |\xi|^{2}\,,\qquad \forall\; v \in \R^{3}, \;\; \xi \in \R^{3}
\end{equation}
for any nonnegative $f \in L^{1}_{2}(\R^{3})$ satisfying $\|f\|_{L^1}=m_{0},$ $\int_{\R^{3}}|v|^{2}f(v)\d v \leq E_{0}$ and 
$\int_{\R^{3}}f(v)\log f(v)\d v \leq H_{0}$.

\medskip

\noindent Assume that $f \in L^{1}_{\gamma+2}(\R^{3})$, then there exists a positive constant $C >0$ depending on $\|f\|_{L^{1}_{\gamma+2}}$ and $\|f\|_{L^1}$ such that
$$a(v)\xi \cdot \xi \leq C\langle v\rangle^{\gamma+2}|\xi|^{2} \qquad \forall \xi \in \R^{3},\:v \in \R^{3}.$$
\end{lem}
\begin{nb}
Notice that 
\begin{equation}
\begin{cases}
\label{eq:bc}
\;\;\;\big|\,{b}(v)\,\big| &\leq 2 \langle v\rangle^{\gamma+1}\|f\|_{L^{1}_{\gamma+1}} \leq 2 \langle v\rangle^{\gamma+1}\|f\|_{L^{1}_{2}}\,,\\
\big|\,\nabla \cdot b(v)\,\big| &\leq 8 \langle v\rangle^{\gamma}\|f\|_{L^{1}_{\gamma}} \leq 8  \langle v\rangle^{\gamma}\|f\|_{L^{1}_{2}}\,, \\
\end{cases}
\end{equation}
since $0 \leq \gamma \leq 1$.
\end{nb}
Here, $f(t,v)$ will denote a weak solution to \eqref{Lce-1} associated to an initial datum $f_{0}$ with mass $m_{0},$ energy $E_{0}$ and entropy $H_{0}$. 
One has then the following  result about propagation and appearance of moments, see \cite[Theorem 3]{DVLandau}.
\begin{lem}\label{lem:momentL} For any $s \geq 0$, 
$$\int_{\R^{3}}\langle v\rangle^{s}f_{0}(v)\d v < \infty \quad \Longrightarrow \quad \sup_{t\geq0}\int_{\R^{3}}\langle v\rangle^{s}f(t,v)\d v < \infty.$$
Moreover, for any $t_{0} >0$ and any $s >0$ there exists $C >0$ depending only on $m_{0},E_{0},H_{0},s$ and $t_{0}$ such that
$$\sup_{t\geq t_{0}}\int_{\R^{3}}\langle v\rangle^{s}f(t,v)\d v \leq C.$$
\end{lem}
\noindent
We have then the following result about instantaneous appearance and uniform bounds for regularity, see \cite[Theorem 5]{DVLandau}.
\begin{lem}\label{lem:soboL}
For any $t_{0} >0$, any integer $k \in \N$ and $s >0$, there exists a constant $C_{t_{0}} >0$ depending only on $m_{0},E_{0},H_{0},k,s$ and $t_{0} >0$ such that
$$\sup_{t\geq t_{0}}\|f(t)\|_{H^{k}_{s}} \leq C_{t_{0}}.$$
\end{lem}
\noindent
We end this section with a simple estimate for integral of the type
$$\int_{\R^{d}}\langle v\rangle^{k}f(v)\left|\log f(v)\right|\d v, \qquad k \geq 0\,,$$
yielding to estimate \eqref{eq:LlogL}. Set, for notational simplicity, 
$$m_{k}:=\int_{\R^{d}}\langle v\rangle^{k}\,f(v)\,\d v, \qquad k \geq 0\,.$$
Let us emphasize that, contrary to the previous results of this appendix,  in the following lemma, $d$ is arbitrary and the function $f$ does not denote any more a solution to the Landau equation. 
\begin{lem}\label{lem:LlogL} For any $k \geq 0$ and any $\varepsilon >0$, there exists $C_{k}(\varepsilon) >0$ such that
\begin{equation}\label{eq:Hp}
\int_{\R^{d}}\langle v\rangle^{k}f(v)\left|\log f(v)\right|\d v \leq \int_{\R^{d}}\langle v\rangle^{k}f(v)\log f(v)\d v + 2m_{k+\varepsilon} + C_{k}(\varepsilon)\,.
\end{equation}
Furthermore, for any $\delta >0$ and any $\varepsilon >0$, there exist $K_{k}(\delta)$ and $C_{k}(\varepsilon)$ such that
\begin{multline}\label{eq:HHp}
\int_{\R^{d}}\langle v\rangle^{k}f(v)\left|\log f(v)\right|\d v \leq \delta\,\int_{\R^{d}}\langle v\rangle^{k}\left|\nabla \sqrt{f}\right|^{2}\d v + K_{k}(\delta)(1+|\log m_{k}|)m_{k} 
+  2m_{k+\varepsilon}+ C_{k}(\varepsilon)\,.
\end{multline}
\end{lem}
\begin{proof} Given $k \geq 0$, we denote by 
$$\H_{k}(f)=\int_{\R^{d}} f(v)\log f(v)\langle v\rangle^{k}\d v, \qquad \mathbf{H}_{k}(f)=\int_{\R^{d}}\langle v\rangle^{k} f(v)|\log f(v)|\d v\,.$$
We set $A=\{v \in \R^{d}\,,\,f(v) < 1\}$,  $A^c=\{v \in \R^{d}\,,\,f(v) \geq 1\}$ so that
\begin{equation*}
\begin{split}
\mathbf{H}_{k}(f)&=\int_{A^{c}} f(v)\log f( v)\langle v\rangle^{k}\d v- \int_{A } f( v)\log f( v)\langle v\rangle^{k}\d v=\H_{k}(f) -2\int_{A} f(v)\log f(v)\langle v\rangle^{k}\d v\\
&=\H_{k}(f) + 2 \int_{A} f(v)\log \left(\frac{1}{f(v)}\right)\langle v\rangle^{k}\d v \,.
\end{split}
\end{equation*}
Given $\varepsilon >0$, set now $B =\{v \in \R^{d}\,;\,f( v) \geq \exp(-\langle v\rangle^{\varepsilon})\}$. If $v \in A  \cap B$, then $\log(\tfrac{1}{f(v)}) \leq \langle v\rangle^{\varepsilon}$ and
\begin{equation*}
\mathbf{H}_{k}(f)\leq \H_{k}(f) + 2m_{k+\varepsilon}+ 2 \int_{A \cap B^c} f( v)\log \left(\frac{1}{f( v)}\right)\langle v\rangle^{k}\d v.
\end{equation*}
Now, since $x\log(1/x) \leq \frac{2}{e} \, \sqrt{x}$ for any  $x \in (0,1)$, we get
$$\int_{A\cap B^c} f(v)\log \left(\frac{1}{f(v)}\right)\langle v\rangle^{k}\d v  \leq \frac{2}{e}\int_{\R^{d}}\exp\left(-\frac{\langle v\rangle^{\varepsilon}}{2}\right)\langle v\rangle^{k}\d v=:C_{k}(\varepsilon) < \infty,$$
which gives \eqref{eq:Hp}. Now, setting $g^{2}(v)=\langle v\rangle^{k}f(v)$, one sees that 
$$\mathcal{H}_{k}(f)=\int_{\R^{d}}g^{2}(v)\log g^{2}(v)\d v - k \int_{\R^{d}}g^{2}(v)\log \langle v\rangle\d v \leq \int_{\R^{d}}g^{2}(v)\log g^{2}(v)\d v$$
since $\langle v\rangle \geq 1.$ We can invoke now the Euclidian logarithmic Sobolev inequality \cite[Theorem 8.14]{Lieb}
$$\int_{\R^{d}}g^{2}\log \frac{g^{2}}{\|g\|^{2}_{L^2}}\d v+d\big(1+\tfrac{1}{2}\log \delta\big)\|g\|_{L^2}^{2} \leq \frac{\delta}{\pi}\int_{\R^{d}}\left|\nabla g\right|^{2}\d v,\qquad \forall\, \delta >0$$
to obtain, observe that $\|g\|^{2}_{L^2}=m_{k}$,
$$\mathcal{H}_{k}(f) \leq \frac{\delta}{\pi}\int_{\R^{d}}\left|\nabla g\right|^{2}\d v + m_{k}\log m_{k} -d\big(1+\tfrac{1}{2}\log \delta\big)m_{k}, \qquad \forall \delta >0\,.$$
Furthermore, there exists $C_{k} >0$ such that 
$$\int_{\R^{d}}\left|\nabla g\right|^{2}\d v=\int_{\R^{d}}\left|\nabla\left( \langle v\rangle^{\frac{k}{2}}\sqrt{f(v)}\right)\right|^{2}\d v \leq C_{k}\left(\int_{\R^{d}}\langle v\rangle^{k}|\nabla \sqrt{f}|^{2}\d v + m_{k}\right)$$
from which we get the result.
\end{proof}

\end{document}